\documentclass[11pt]{article}%
\usepackage{amsmath,amsfonts,amsthm,amssymb, graphicx}
\usepackage{subfig}
\usepackage{amsmath}
\usepackage{amsfonts}
\usepackage{amssymb}
\usepackage{graphicx}%
\providecommand{\U}[1]{\protect\rule{.1in}{.1in}}
\pdfoutput=1
\theoremstyle{plain}
\newtheorem{thm}{Theorem}

\newtheorem{acknowledgement}{Acknowledgement}

\begin{document}

\title{On the Limiting Ratio of Current Age to Total Life for Null Recurrent Renewal Processes}
\author{Blanchet, J., Glynn, P. and Thorisson, H.}
\maketitle

\begin{abstract}
If the inter-arrival time distribution of a renewal process is regularly
varying with index $\alpha\in\left(  0,1\right)  $ (i.e. the inter-arrival
times have infinite mean) and if $A\left(  t\right)  $ is the associated age
process at time $t$. Then we show that if $C\left(  t\right)  $ is the length
of the current cycle at time $t$,
\[
A\left(  t\right)  /C\left(  t\right)  \Rightarrow U^{1/\alpha},
\]
where $U$ is $U\left(  0,1\right)  $. This extends a classical result in
renewal theory in the finite mean case which indicates that the limit is
$U\left(  0,1\right)  $.

\end{abstract}

\section{The Result}

In this note we revisit some classical renewal theorems for infinite mean
inter-arrival time distributions. Consider a sequence of i.i.d. non-arithmetic
non-negative random variables $\{X_{n}:n\geq1\}$. Set $S_{0}=0$ and define
$S_{n}=X_{1}+...+X_{n}$ for $n\geq1$. Further, consider the associated renewal
process%
\[
N\left(  t\right)  =\max\{n\geq0:S_{n}\leq t\},
\]
and the corresponding age, residual life-time, and cycle-in-progress processes
defined as%
\[
A\left(  t\right)  :=t-S_{N\left(  t\right)  },\text{ \ }B\left(  t\right)
=S_{N\left(  t\right)  +1}-t,\text{ \ \ }C\left(  t\right)  =B\left(
t\right)  +A\left(  t\right)  .
\]
It is well known that if $EX_{n}<\infty$, then $\left(  A(t),C\left(
t\right)  \right)  \Rightarrow\left(  CU,C\right)  $, where
\[
P(C\in A)=E\left(  X_{1}I\left(  X_{1}\in A\right)  \right)  /E\left(
X_{1}\right)  ,
\]
and $U\sim U\left(  0,1\right)  $, see for example, Asmussen (2003).

Our goal here is to investigate the case in which $EX_{n}=\infty$. In
particular, we assume that $X_{n}$'s have a regularly varying density at
infinity with index $\alpha+1$ and $\alpha\in\left(  0,1\right)  $; that is,
assume that there exists $t_{0}>0$ such that for all $t>t_{0}$
\[
\bar{F}\left(  t\right)  :=P\left(  X_{n}>t\right)  =\int_{t}^{\infty
}s^{-\alpha-1}L\left(  s\right)  ds,
\]
for a slowly varying function $L\left(  \cdot\right)  $. We will prove the
following theorem.

\begin{thm}%
\[
A\left(  t\right)  /C\left(  t\right)  \Longrightarrow U^{1/\alpha}%
\]
as $t\rightarrow\infty$.
\end{thm}

\begin{proof}
First we obtain a renewal equation for the distribution of the regenerative
process $V\left(  t\right)  =A\left(  t\right)  /C\left(  t\right)  $, namely%
\begin{align*}
a\left(  t\right)   &  =P\left(  V\left(  t\right)  >x\right)  \\
&  =\int_{t}^{\infty}P\left(  V\left(  t\right)  >x|\tau_{1}=s\right)
P\left(  \tau_{1}\in ds\right)  +\int_{0}^{t}a\left(  t-s\right)  P\left(
\tau_{1}\in ds\right)  \\
&  =\int_{t}^{\infty}I\left(  t/s>x\right)  P\left(  \tau_{1}\in ds\right)
+\int_{0}^{t}a\left(  t-s\right)  P\left(  \tau_{1}\in ds\right)  \\
&  =\int_{t}^{t/x}P\left(  \tau_{1}\in ds\right)  +\int_{0}^{t}a\left(
t-s\right)  P\left(  \tau_{1}\in ds\right)  \\
&  =b\left(  t\right)  +\int_{0}^{t}a\left(  t-s\right)  F\left(  ds\right)  ,
\end{align*}
where $b\left(  t\right)  =\bar{F}\left(  t\right)  -\bar{F}\left(
t/x\right)  $. We then conclude that
\[
a\left(  t\right)  =\int_{0}^{t}b\left(  t-s\right)  u\left(  ds\right)  ,
\]
with $u\left(  s\right)  =E(N\left(  s\right)  +1)$ being the renewal
function. We have from Theorem 5 of Erickson (1970) that
\begin{equation}
u\left(  t\right)  \sim c^{\ast}/\bar{F}\left(  t\right)  .\label{E_R}%
\end{equation}
(This property actually holds true even if $\alpha\in\lbrack0,1]$.) We then
need to evaluate the limit of
\[
a\left(  t\right)  =\int_{0}^{1}\bar{F}\left(  t(1-r)\right)  u\left(
tdr\right)  -\int_{0}^{1}\bar{F}\left(  t(1-r)/x\right)  u\left(  tdr\right)
\]
as $t\rightarrow\infty$. We shall argue the (quite intuitive, due to
(\ref{E_R})) limits,
\begin{equation}
\int_{0}^{1}\bar{F}\left(  t(1-r)\right)  u\left(  tdr\right)  \sim c^{\ast
}\int_{0}^{1}\frac{\bar{F}\left(  t(1-r)\right)  }{\bar{F}\left(  t\right)
}\cdot\frac{u\left(  tdr\right)  }{u\left(  t\right)  }\sim c^{\ast}\alpha
\int_{0}^{1}(1-r)^{-\alpha}r^{\alpha-1}dr,\label{1}%
\end{equation}
and similarly%
\begin{equation}
\int_{0}^{1}\bar{F}\left(  t(1-r)/x\right)  u\left(  tdr\right)  \sim c^{\ast
}\alpha x^{\alpha}\int_{0}^{1}(1-r)^{-\alpha}r^{\alpha-1}dr,\label{2}%
\end{equation}
thereby concluding that
\[
P\left(  V\left(  t\right)  >x\right)  \sim c^{\ast}\alpha(1-x^{\alpha}).
\]
By tightness have that $c^{\ast}\alpha=1$ and hence, provided that (\ref{1})
and (\ref{2}) we will be able to conclude that
\[
P\left(  V\left(  t\right)  >x\right)  \sim(1-x^{\alpha})=P\left(
U^{1/\alpha}>x\right)
\]
as $t\rightarrow\infty$. Theorem 2 in Teugels (1968) actually indicates that
the asymptotics in (\ref{1}) and (\ref{2}) are indeed correct, and for this
part it is important to assume that $\alpha\in\left(  0,1\right)  $ and that
the slowly varying part component of $\bar{F}\left(  \cdot\right)  $ satisfies
some mild conditions (which are satisfied precisely if $X_{n}$ has a slowly
varying density as we have assumed).
\end{proof}

\bigskip

\textbf{Remark 1:} It is desirable to show the result for $\alpha\in
\lbrack0,1]$ and we would like to remove the conditions on the slowly varying
assumption on $\bar{F}\left(  \cdot\right)  $ in Teugels (1968). There are
results obtained for

\bigskip

\textbf{Remark 2: }We finish this note with a comment. It turns out that the
joint distribution of $(A\left(  t\right)  ,B\left(  t\right)  )/t$ as
$t\rightarrow\infty$ was derived by Dynkin (1955) and Lamperti (1962), (see,
for example, Theorem 8.6.3 in Bingham, Goldie and Teugels (1987)). In
principle, the law obtained above might be derived from the Dynkin-Lamperti
theorem, but such derivation does not appear to be as direct as our derivation
above. Maybe this explains why the simple asymptotic limit obtained for
$A\left(  t\right)  /C\left(  t\right)  $ as $t\rightarrow\infty$ appears to
not have been explicitly identified in well known references for the
Dynkin-Lamperti theorem.

\begin{acknowledgement}
This research was partially supported by the grants DMS-0806145, CMMI-0846816
and CMMI-1069064. This result was presented as a solution to an open problem
discussed during the Stochastic Networks Conference organized at the Newton
Institute in 2013.
\end{acknowledgement}

\end{document}